\documentclass[a4paper]{amsart}

\usepackage[utf8]{inputenc}
\usepackage[english]{babel}

\usepackage[T1]{fontenc}
\usepackage{lmodern}  % Latine Modern

\usepackage{my-math-symbol}
\usepackage{my-cleveref}
\usepackage{my-equation-numbering}
\usepackage{comment}
\usepackage{tikz}
\usetikzlibrary{cd}

\newcommand{\GJMS}{P}
\newcommand{\dilation}{\delta}

\newcommand{\Hsymbol}{S_{H}}
\newcommand{\Hpsido}{\Psi_{H}}
\newcommand{\HSobolev}{W_{H}}
\newcommand{\ovone}{\overline{1}}

\usepackage[non-sorted-cites,initials,alphabetic,nobysame]{amsrefs}

\AtBeginDocument{%
	\def\MR#1{}
}

\title{CR Paneitz operator on non-embeddable CR manifolds}
\author{Yuya Takeuchi}
\address{Division of Mathematics \\ Institute of Pure and Applied Sciences \\ University of Tsukuba
	\\ 1- 1- 1 Tennodai, Tsukuba, Ibaraki 305-8571 Japan}
\email{ytakeuchi@math.tsukuba.ac.jp, yuya.takeuchi.math@gmail.com}



\thanks{This work was supported by JSPS KAKENHI Grant Number JP21K13792.}

\begin{document}

\begin{abstract}
	The CR Paneitz operator is closely related to some important problems in CR geometry.
	In this paper,
	we consider this operator on a non-embeddable CR manifold.
	This operator is essentially self-adjoint
	and its spectrum is discrete except zero.
	Moreover,
	the eigenspace corresponding to each non-zero eigenvalue
	is a finite dimensional subspace of the space of smooth functions.
	Furthermore,
	we show that the CR Paneitz operator on the Rossi sphere,
	an example of non-embeddable CR manifolds,
	has infinitely many negative eigenvalues,
	which is significantly different from the embeddable case.
\end{abstract}

\maketitle

\section{Introduction}
\label{section:introduction}

The \emph{CR Paneitz operator} plays a crucial role in CR geometry of dimension theree.
For example,
this operator appears in the transformation law of the logarithmic singularity of the \Szego kernel~\cite{Hirachi1993},
which is also known as the CR $Q$-curvature~\cite{Fefferman-Hirachi2003}.
Moreover,
the non-negativity of the CR Paneitz operator is deeply connected to
global embeddability~\cites{Chanillo-Chiu-Yang2012,Takeuchi2020-Paneitz}
and the CR positive mass theorem~\cite{Cheng-Malchiodi-Yang2017},
which has an application to the CR Yamabe problem.

Let $(M, T^{1,0}M, \theta)$ be a closed pseudo-Hermitian manifold of dimension three.
We consider the CR Paneitz operator $P$ on $M$
as an unbounded operator on $L^{2}(M)$ with domain $C^{\infty}(M)$,
which is closable.
In the embeddable case,
Hsiao~\cite{Hsiao2015} has studied analytic properties of this operator;
see also \cite{Takeuchi2023-GJMS} for some improvements and generalizations to higher dimensions.
The CR Paneitz operator on an embeddable CR manifold is essentially self-adjoint and has closed range.
Moreover,
its spectrum is a discrete subset of $\bbR$ and consists only of eigenvalues.
Furthermore,
the eigenspace associated with each non-zero eigenvalue
is a finite-dimensional subspace of $C^{\infty}(M)$.

It is natural to ask what happens in the non-embeddable case.
We will first prove that $P$ is essentially self-adjoint
even in the non-embeddable case.

\begin{theorem}
\label{thm:ess-self-adj-of-CR-Paneitz}
	The CR Paneitz operator $P$ is essentially self-adjoint;
	equivalently,
	the maximal closed extension of $\GJMS$ is self-adjoint.
\end{theorem}

We use the same letter $\GJMS$
for the maximal closed extension of the CR Paneitz operator
by abuse of notation.
Let $E$ be the resolution of the identity for $\GJMS$,
and set
\begin{equation}
	\pi_{\lambda} \coloneqq E(\clcl{- \lambda}{\lambda})
	\colon L^{2}(M) \to \Dom \GJMS
\end{equation}
for $\lambda \geq 0$.
Note that $\pi_{\lambda}$ is an orthogonal projection of $L^{2}(M)$.
We will show that if $\lambda > 0$,
then $\GJMS \pi_{\lambda}$ is a smoothing operator (\cref{thm:smoothing-property-of-projection}).
As an application,
we study the spectrum $\Spec \GJMS$ of $\GJMS$.
As opposed to the embeddable case,
we can only show that $\Spec \GJMS$ is discrete \emph{except} zero.

\begin{theorem}
\label{thm:spectrum-of-CR-Paneitz}
	The set $\Spec \GJMS \setminus \{0\}$ is a discrete subset of $\bbR \setminus \{0\}$
	and consists only of eigenvalues of finite multiplicity.
	Moreover,
	any eigenfunction of $\GJMS$ corresponding to each non-zero eigenvalue is smooth.
\end{theorem}

We will also show that $\pi_{\lambda}$ is a Heisenberg pseudodifferential operator
(\cref{thm:spectral-projection-is-Hpsido}).
In particular,
the singular support of its Schwartz kernel
is contained in the diagonal $\Delta_{M} \subset M \times M$.
Our proof is inspired by that of \cite{Hsiao-Marinescu2017}.
Remark that $\pi_{0}$ is \emph{not} necessarily a Heisenberg pseudodifferential operator;
see \cref{rem:ortho-proj-to-Ker-is-not-Hpsido}.

We also study analytic properties of the CR Paneitz operator on the Rossi sphere,
a homogeneous non-embeddable strictly pseudoconvex CR manifold;
see \cite{Chen-Shaw2001} for example.
The unit sphere
\begin{equation}
	S^{3}
	\coloneqq \Set{(z, w) \in \bbC^{2} | \abs{z}^{2} + \abs{w}^{2} = 1}
\end{equation}
has the canonical CR structure $T^{1, 0} S^{3}$.
This CR structure is spanned by
\begin{equation}
	Z_{1}
	\coloneqq \ovw \frac{\del}{\del z} - \ovz \frac{\del}{\del w}.
\end{equation}
A canonical contact form $\theta$ on $S^{3}$ is given by
\begin{equation}
	\theta
	\coloneqq \frac{\sqrt{- 1}}{2} (z d \ovz + w d \ovw
		- \ovz d z - \ovw d w)|_{S^{3}}.
\end{equation}
For a real number $0 < \abs{t} < 1$,
the \emph{Rossi sphere} $(S^{3}_{t}, T^{1, 0} S^{3}_{t})$ is defined by
\begin{equation}
	(S^{3}_{t}, T^{1, 0} S^{3}_{t})
	\coloneqq (S^{3}, \bbC (Z_{1} + t \overline{Z_{1}})).
\end{equation}
Denote by $P(t)$ the CR Paneitz operator with respect to $(S^{3}_{t}, T^{1,0}S^{3}_{t}, \theta)$.
We can apply the theory of spherical harmonics
since $P(t)$ is invariant under $U(2)$-action.
Similar to the method in \cite{Abbas-Brown-Ramasami-Zeytuncu2019},
we can prove the following

\begin{theorem}
\label{thm:infinitely-many-negative-eigenvalue}
	The CR Paneitz operator $P(t)$ on the Rossi sphere $(S^{3}_{t}, T^{1, 0} S^{3}_{t}, \theta)$
	has infinitely many negative eigenvalues counted without multiplicity.
\end{theorem}

This paper is organized as follows.
In \cref{section:CR-manifolds},
we recall basic facts on CR manifolds and the definition of the CR Paneitz operator.
\cref{section:Heisenberg-calculus} gives a brief exposition of the Heisenberg calculus,
which is a main tool for studying the CR Paneitz operator.
\cref{section:proof-of-main-theorems} is devoted to proofs of
\cref{thm:ess-self-adj-of-CR-Paneitz,thm:spectrum-of-CR-Paneitz} as applications of Heisenberg calculus.
In \cref{section:CR-Paneitz-operator-on-Rossi-spheres},
we study the CR Paneitz operator on the Rossi sphere via spherical harmonics.
In \cref{section:concluding-remarks},
we propose some related problems and give some observations.

\section{CR manifolds}
\label{section:CR-manifolds}

Let $M$ be a smooth three-dimensional manifold without boundary.
A \emph{CR structure} is a complex line subbundle $T^{1, 0} M$
of the complexified tangent bundle $TM \otimes \mathbb{C}$ such that
\begin{equation}
	T^{1, 0}M \cap T^{0, 1}M = 0,
\end{equation}
where $T^{0, 1} M$ is the complex conjugate of $T^{1, 0} M$ in $T M \otimes \mathbb{C}$.
Introduce an operator $\delbb \colon C^{\infty}(M) \to \Gamma((T^{0, 1} M)^{\ast})$ by
\begin{equation}
	\delbb f = (d f)|_{T^{0, 1} M}.
\end{equation}
A CR manifold $(M, T^{1, 0} M)$ is said to be \emph{embeddable}
if there exists a smooth embedding $F$ from $M$ into some $\bbC^{N}$
such that $F_{\ast} T^{1, 0} M \subset T^{1, 0} \bbC^{N}$.

A CR structure $T^{1, 0} M$ is said to be \emph{strictly pseudoconvex}
if there exists a nowhere-vanishing real one-form $\theta$ on $M$
such that
$\theta$ annihilates $T^{1, 0} M$ and
\begin{equation}
	- \sqrt{- 1} d \theta (Z, \ovZ) > 0, \qquad
	0 \neq Z \in T^{1, 0} M;
\end{equation}
we call such a one-form a \emph{contact form}.
The triple $(M, T^{1, 0} M, \theta)$ is called a \emph{pseudo-Hermitian manifold}.
Denote by $T$ the \emph{Reeb vector field} with respect to $\theta$; 
that is, the unique vector field satisfying
\begin{equation}
	\theta(T) = 1, \qquad T \contr d\theta = 0.
\end{equation}
Let $Z_{1}$ be a local frame of $T^{1, 0} M$,
and set $Z_{\ovone} = \overline{Z_{1}}$.
Then
$(T, Z_{1}, Z_{\ovone})$ gives a local frame of $T M \otimes \mathbb{C}$,
called an \emph{admissible frame}.
Its dual frame $(\theta, \theta^{1}, \theta^{\ovone})$
is called an \emph{admissible coframe}.
The two-form $d \theta$ is written as
\begin{equation}
	d \theta = \sqrt{- 1} l_{1 \ovone} \theta^{1} \wedge \theta^{\ovone},
\end{equation}
where $l_{1 \ovone}$ is a positive function.
We use $l_{1 \ovone}$ and its multiplicative inverse $l^{1 \ovone}$
to raise and lower indices.

A contact form $\theta$ induces a canonical connection $\nabla$,
called the \emph{Tanaka-Webster connection} with respect to $\theta$.
It is defined by
\begin{equation}
	\nabla T
	= 0,
	\quad
	\nabla Z_{1}
	= {\omega_{1}}^{1} Z_{1},
	\quad
	\nabla Z_{\ovone}
	= {\omega_{\ovone}}^{\ovone} Z_{\ovone}
	\quad
	\rbra*{ {\omega_{\ovone}}^{\ovone}
	= \overline{{\omega_{1}}^{1}} }
\end{equation}
with the following structure equations:
\begin{gather}
	d \theta^{1}
	= \theta^{1} \wedge {\omega_{1}}^{1}
	+ {A^{1}}_{\ovone} \theta \wedge \theta^{\ovone}, \\
	d l_{1 \ovone}
	= {\omega_{1}}^{1} l_{1 \ovone}
	+ l_{1 \ovone} {\omega_{\ovone}}^{\ovone}.
\end{gather}
The tensor $A_{1 1} = \overline{A_{\ovone \ovone}}$
is called the \emph{Tanaka-Webster torsion}.
The curvature form
${\Omega_{1}}^{1} = d {\omega_{1}}^{1}$
of the Tanaka-Webster connection satisfies
\begin{equation} \label{eq:curvature-form-of-TW-connection}
	{\Omega_{1}}^{1}
	\equiv \Scal \cdot l_{1 \ovone} \theta^{1} \wedge \theta^{\ovone}
		\qquad \text{modulo } \theta,
\end{equation}
where $\Scal$ is the \emph{Tanaka-Webster scalar curvature}.
We denote the components of a successive covariant derivative of a tensor
by subscripts preceded by a comma,
for example, $K_{1 \ovone , 1}$;
we omit the comma if the derivatives are applied to a function.
We use the index $0$ for the component $T$ or $\theta$ in our index notation.
In this notation,
the operator $\delbb$ is given by
\begin{equation}
\label{eq:tensorial-rep-of-delb}
	\delbb f
	= f_{\ovone} \theta^{\ovone}.
\end{equation}
The commutators of the second derivatives for $u \in C^{\infty}(M)$ are given by
\begin{equation}
\label{eq:commutator-of-covariant-derivatives}
	u_{1 \ovone} - u_{\ovone 1}
	= \sqrt{- 1} l_{1 \ovone} u_{0},
	\qquad
	u_{0 1} - u_{1 0}
	= A_{1 1} u^{1};
\end{equation}
see~\cite{Lee1988}*{(2.14)}.
Define the \emph{Kohn Laplacian} $\Box_{b}$ and the \emph{sub-Laplacian} $\Delta_{b}$ by
\begin{equation}
\label{eq:Kohn-Laplacian}
	\Box_{b} u
	\coloneqq - {u_{\ovone}}^{\ovone},
	\qquad
	\Delta_{b} u
	\coloneqq (\Box_{b} + \overline{\Box}_{b}) u
\end{equation}
for $u \in C^{\infty}(M)$.
It follows from \cref{eq:commutator-of-covariant-derivatives} that
\begin{equation}
\label{eq:complex-conjugate-of-Kohn-Laplacian}
	\overline{\Box}_{b}
	= \Box_{b} - \sqrt{- 1} T.
\end{equation}
It is known that $(M, T^{1, 0} M)$ is embeddable
if and only if the Kohn Laplacian $\Box_{b}$ has closed range~\cites{Boutet_de_Monvel1975,Kohn1986}.

The \emph{CR Paneitz operator} $\GJMS$ is the fourth-order differential operator given by
\begin{equation}
	\GJMS
	\coloneqq \overline{\Box}_{b} \Box_{b} + \calQ,
\end{equation}
where
\begin{equation}
	\calQ u
	\coloneqq \sqrt{- 1} (A^{\ovone \ovone} u_{\ovone})_{, \ovone}.
\end{equation}
This operator is real and formally self-adjoint;
see \cite{Gover-Graham2005}*{Proposition 5.1} for example.
Note that our $\GJMS$ is just the operator $P_{0, 0}$ in this paper.
Define an operator $\calC \colon \Gamma((T^{0, 1} M)^{\ast}) \to C^{\infty}(M)$ by
\begin{equation}
	\calC(\tau_{\ovone} \theta^{\ovone})
	\coloneqq - \overline{\Box}_{b} ({\tau_{\ovone ,}}^{\ovone})
		+ \sqrt{- 1} (A^{\ovone \ovone} \tau_{\ovone})_{, \ovone}.
\end{equation}
It follows from \cref{eq:tensorial-rep-of-delb,eq:Kohn-Laplacian} that $\GJMS = \calC \delb_{b}$.
It is known that
the CR Paneitz operator is non-negative
if $(M, T^{1, 0} M)$ is embeddable~\cite{Takeuchi2020-Paneitz}*{Theorem 1.1}.
Conversely,
$(M, T^{1, 0} M)$ is embeddable
if the CR Paneitz operator is non-negative and the Tanaka-Webster scalar curvature is positive%
~\cite{Chanillo-Chiu-Yang2012}*{Theorem 1.4(a)}.

\section{Heisenberg calculus}
\label{section:Heisenberg-calculus}

In this section,
we recall basic properties of Heisenberg pseudodifferential operators;
see~\cites{Beals-Greiner1988,Ponge2008-Book}
for a comprehensive introduction to the Heisenberg calculus.

Throughout this section,
we fix a closed pseudo-Hermitian manifold $(M, T^{1, 0} M, \theta)$ of dimension three.
Set
\begin{equation}
	\frakg M
	\coloneqq (T M / H M) \oplus H M.
\end{equation}
The Reeb vector field $T$ defines a nowhere-vanishing section $[T]$ of $T M / H M$.
For sections $X_{0}$ and $Y_{0}$ of $T M / H M$
and $X^{\prime}$ and $Y^{\prime}$ of $H M$,
the Lie bracket $\comm{X_{0} + X^{\prime}}{Y_{0} + Y^{\prime}}$ is defined by
\begin{equation}
	\comm{X_{0} + X^{\prime}}{Y_{0} + Y^{\prime}}
	\coloneqq - d \theta (X^{\prime}, Y^{\prime}) [T].
\end{equation}
This bracket makes $\frakg M$ a bundle of two-step nilpotent Lie algebras.
The dilation $\dilation_{r}$ on $\frakg M$ is defined by
\begin{equation}
	\dilation_{r} |_{T M / H M}
	\coloneqq r^{2},
	\qquad
	\dilation_{r} |_{H M}
	\coloneqq r.
\end{equation}
For $m \in \bbZ$,
the space $\Hsymbol^{m}(M)$
consists of functions in $C^{\infty}((\frakg M)^{\ast} \setminus \{0\})$
that are homogeneous of degree $m$ on each fiber.
This space has a bilinear product
\begin{equation}
	\ast \colon \Hsymbol^{m_{1}}(M) \times \Hsymbol^{m_{2}}(M)
	\to \Hsymbol^{m_{1} + m_{2}}(M).
\end{equation}

For $m \in \bbZ$,
denote by $\Hpsido^{m}(M)$
the space of \emph{Heisenberg pseudodifferential operators
$A \colon C^{\infty}(M) \to C^{\infty}(M)$ of order $m$}.
This space is closed under sum, complex conjugate, transpose, and formal adjoint%
~\cite{Ponge2008-Book}*{Proposition 3.1.23}.
In particular,
any $A \in \Hpsido^{m}(M)$ extends to a linear operator
\begin{equation}
	A \colon \scrD^{\prime}(M) \to \scrD^{\prime}(M),
\end{equation}
where $\scrD^{\prime}(M)$ is the space of distributions on $M$.
For example,
$Z \in \Gamma(T^{1, 0} M)$ is an element of $\Hpsido^{1}(M)$
and $T \in \Hpsido^{2}(M)$.
Note that $\Hpsido^{- \infty}(M) \coloneqq \bigcap_{m \in \bbZ} \Hpsido^{m}(M)$
coincides with the space of smoothing operators on $M$.
As in the usual pseudodifferential calculus,
there exists the Heisenberg principal symbol,
which has some good properties:

\begin{proposition}[\cite{Ponge2008-Book}*{Propositions 3.2.6 and 3.2.9}]
\label{prop:Heisenberg-principal-symbol}
	(i) The Heisenberg principal symbol $\sigma_{m}$ gives the following exact sequence:
	\begin{equation}
		0 \to \Hpsido^{m - 1}(M) \hookrightarrow \Hpsido^{m}(M)
			\xrightarrow{\sigma_{m}} \Hsymbol^{m}(M) \to 0.
	\end{equation}

	(ii) For $A_{1} \in \Hpsido^{m_{1}}(M)$ and $A_{2} \in \Hpsido^{m_{2}}(M)$,
	the operator $A_{1} A_{2}$ is a Heisenberg pseudodifferential operator of order $m_{1} + m_{2}$,
	and
	\begin{equation}
		\sigma_{m_{1} + m_{2}}(A_{1} A_{2}) = \sigma_{m_{1}}(A_{1}) \ast \sigma_{m_{2}}(A_{2}).
	\end{equation}
\end{proposition}

On the other hand,
there exists a crucial difference between the usual pseudodifferential calculus and the Heisenberg one.
Since the product of Heisenberg principal symbol is non-commutative,
the commutator $\comm{A_{1}}{A_{2}}$ of $A_{1} \in \Hpsido^{m_{1}}(M)$ and $A_{2} \in \Hpsido^{m_{2}}(M)$
is not an element of $\Hpsido^{m_{1} + m_{2} - 1}(M)$ in general.
However,
we have the following

\begin{lemma}[\cite{Takeuchi2023-GJMS}*{Lemma 4.2}]
\label{lem:commutator-for-Reeb-vector-field}
	If $A \in \Hpsido^{m}(M)$,
	then $\comm{T}{A} \in \Hpsido^{m + 1}(M)$.
\end{lemma}

Next we consider approximate inverses of Heisenberg pseudodifferential operators.
We write $A \sim B$ if $A - B$ is a smoothing operator.
Let $A \in \Hpsido^{m}(M)$.
An operator $B \in \Hpsido^{- m}(M)$ is called a \emph{parametrix} of $A$
if $A B \sim I$ and $B A \sim I$.
The existence of a parametrix of a Heisenberg pseudodifferential operator
is determined only by its Heisenberg principal symbol.

\begin{proposition}[\cite{Ponge2008-Book}*{Proposition 3.3.1}]
\label{prop:equivalent-conditions-for-existence-of-parametrix}
	Let $A \in \Hpsido^{m}(M)$
	with Heisenberg principal symbol $a \in \Hsymbol^{m}(M)$.
	Then the following are equivalent:
	\begin{enumerate}
		\item $A$ has a parametrix;
		\item there exists $B \in \Hpsido^{- m}(M)$ such that
			$A B - I, B A - I \in \Hpsido^{- 1}(M)$;
		\item there exists $b \in \Hsymbol^{-m}(M)$ such that
			$a \ast b = b \ast a = 1$.
	\end{enumerate}
\end{proposition}

Now consider the Heisenberg differential operator $\Delta_{b} + I$ of order $2$.
It is known that this operator has a parametrix;
see the proof of~\cite{Ponge2008-Book}*{Proposition 3.5.7} for example.
Since $\Delta_{b} + I$ is positive and self-adjoint,
the $k / 2$-th power $(\Delta_{b} + I)^{k / 2}$ of $\Delta_{b} + I$,
$k \in \bbZ$,
is a Heisenberg pseudodifferential operator of order $k$~\cite{Ponge2008-Book}*{Theorems 5.3.1 and 5.4.10}.
Using this operator,
we define
\begin{equation}
	\HSobolev^{k}(M)
	:= \Set{ u \in \scrD^{\prime}(M) \mid (\Delta_{b} + I)^{k / 2} u \in L^{2}(M) }.
\end{equation}
This space is a Hilbert space with the inner product
\begin{equation}
	\iproduct{u}{v}_{k}
	= \iproduct{(\Delta_{b} + I)^{k / 2} u}{(\Delta_{b} + I)^{k / 2} v}_{L^{2}}.
\end{equation}
The space $C^{\infty}(M)$ is dense in $\HSobolev^{k}(M)$ for any $k \in \bbZ$,
and we have
\begin{equation}
	C^{\infty}(M) = \bigcap_{k \in \bbZ} \HSobolev^{k}(M),
	\qquad
	\scrD^{\prime}(M) = \bigcup_{k \in \bbZ} \HSobolev^{k}(M)
\end{equation}
as topological vector spaces~\cite{Ponge2008-Book}*{Proposition 5.5.3}.
Remark that
the Hilbert space $\HSobolev^{k}(M)$ coincides with the Folland-Stein space $S^{k, 2}(M)$
as a topological vector space~\cite{Ponge2008-Book}*{Proposition 5.5.5}.
Heisenberg pseudodifferential operators act on these Hilbert spaces as follows:

\begin{proposition}[\cite{Ponge2008-Book}*{Propositions 5.5.8} and \cite{Takeuchi2023-GJMS}*{Proposition 4.6}]
\label{prop:mapping-properties-of-Hpsido}
	Any $A \in \Hpsido^{m}(M)$ extends to a continuous linear operator
	\begin{equation}
		A \colon \HSobolev^{k + m}(M) \to \HSobolev^{k}(M)
	\end{equation}
	for every $k \in \bbZ$.
	In particular if $m < 0$,
	the operator $A \colon L^{2}(M) \to L^{2}(M)$ is compact.
\end{proposition}

\section{Proof of \cref{thm:ess-self-adj-of-CR-Paneitz,thm:spectrum-of-CR-Paneitz}}
\label{section:proof-of-main-theorems}

Let $(M, T^{1, 0} M, \theta)$ be a closed strictly pseudoconvex CR manifold of dimension three.
If $(M, T^{1, 0} M)$ is embeddable,
then there exist the orthogonal projection to $\Ker \Box_{b}$,
known as the \emph{\Szego projection},
and the partial inverse of $\Box_{b}$~\cite{Beals-Greiner1988}*{Theorem 25.20}.
If $(M, T^{1, 0} M)$ is non-embeddable,
then we can not use the partial inverse of $\Box_{b}$.
However,
we have an \emph{approximate} \Szego projection and partial inverse of $\Box_{b}$.

\begin{theorem}[\cite{Beals-Greiner1988}*{Proposition 25.4 and Corollaries 25.64 and 25.67}]
\label{thm:approximate-Szego-projection}
	There exist $S \in \Hpsido^{0}(M)$ and $N \in \Hpsido^{- 2}(M)$ such that
	\begin{equation}
		\Box_{b} N + S
		\sim N \Box_{b} + S
		\sim I, 
		\qquad
		\Box_{b} S
		\sim S \Box_{b}
		\sim 0,
		\qquad
		S
		\sim S^{\ast}
		\sim S^{2},
		\qquad
		\delb_{b} S
		\sim 0.
	\end{equation}
\end{theorem}

We first show some commutation relations of $S$ and $\ovS$.

\begin{lemma}
\label{lem:commutator-of-Szego-and-Laplacian}
	One has $\comm{S}{\overline{\Box}_{b}}, \comm{\ovS}{\Box_{b}} \in \Hpsido^{1}(M)$.
\end{lemma}

\begin{proof}
	It follows from \cref{thm:approximate-Szego-projection} that
	\begin{equation}
		S \overline{\Box}_{b}
		= S (\Box_{b} - \sqrt{- 1} T)
		\sim - \sqrt{- 1} T S - \sqrt{- 1} \comm{S}{T}
		\sim \overline{\Box}_{b} S - \sqrt{- 1} \comm{S}{T}.
	\end{equation}
	We obtain from \cref{lem:commutator-for-Reeb-vector-field} that
	\begin{equation}
		\comm{S}{\overline{\Box}_{b}}
		\sim - \sqrt{- 1} \comm{S}{T}
		\in \Hpsido^{1}(M).
	\end{equation}
	Taking the complex conjugate yields $\comm{\ovS}{\Box_{b}} \in \Hpsido^{1}(M)$.
\end{proof}

Next we consider the composition $S$ and $\ovS$.

\begin{lemma}
\label{lem:composition-of-Szego-projection}
	One has $S \ovS, \ovS S \in \Hpsido^{- 1}(M)$.
\end{lemma}

\begin{proof}
	It follows from \cref{lem:commutator-of-Szego-and-Laplacian} that
	\begin{equation}
		\Delta_{b} S \ovS
		\sim \overline{\Box}_{b} S \ovS
		= S \overline{\Box}_{b} \ovS + \comm{\overline{\Box}_{b}}{S} \ovS
		\sim \comm{\overline{\Box}_{b}}{S} \ovS
		\in \Hpsido^{1}(M)
	\end{equation}
	Since $\Delta_{b}$ has a parametrix,
	we have $S \ovS \in \Hpsido^{- 1}(M)$.
	Taking the complex conjugate gives
	$\ovS S \in \Hpsido^{- 1}(M)$.
\end{proof}

\begin{remark}
	If $M$ is embeddable and $S$ is the \Szego projection,
	then we have $S \ovS \sim \ovS S \sim 0$;
	see \cite{Hsiao2015}*{Lemma 4.2}.
\end{remark}

Let $\GJMS$ be the CR Paneitz operator on $(M, T^{1, 0} M, \theta)$.
Since $\GJMS = \calC \delb_{b}$,
we have $P S \sim 0$.
Taking the complex conjugate yields $P \ovS \sim 0$.
Hence $\Pi_{0} \coloneqq S + \ovS \in \Hpsido^{0}(M)$ satisfies $\GJMS \Pi_{0} \sim 0$.
Set $G_{0} \coloneqq N \ovN \in \Hpsido^{- 4}(M)$.
It follows from \cref{lem:commutator-of-Szego-and-Laplacian,lem:composition-of-Szego-projection} that
\begin{equation}
	\GJMS G_{0}
	\equiv \overline{\Box}_{b} (I - S) \ovN
	\equiv (I - S) \overline{\Box}_{b} \ovN
	\equiv (I - S) (I - \ovS)
	\equiv I - \Pi_{0}
\end{equation}
modulo $\Hpsido^{- 1}(M)$;
note that $\calQ \in \Hpsido^{2}(M)$.
Thus we have
\begin{equation}
	R_{0}
	\coloneqq \GJMS G_{0} + \Pi_{0} - I \in \Hpsido^{- 1}(M).
\end{equation}

\begin{lemma}
	The operator $I + R_{0} \in \Hpsido^{0}(M)$ has a parametrix $A_{0} \in \Hpsido^{0}(M)$.
	Moreover,
	$A_{0}$ satisfies $A_{0} - I \in \Hpsido^{- 1}(M)$.
\end{lemma}

\begin{proof}
	Since
	\begin{equation}
		I (I + R_{0}) - I
		= (I + R_{0}) I - I
		= R_{0}
		\in \Hpsido^{- 1}(M),
	\end{equation}
	$I + R_{0}$ has a parametrix $A_{0}$ by \cref{prop:equivalent-conditions-for-existence-of-parametrix}.
	We obtain from $R_{0} \in \Hpsido^{-1}(M)$ and \cref{prop:Heisenberg-principal-symbol} that
	\begin{equation}
		\sigma_{0}(A_{0})
		= \sigma_{0}((I + R_{0}) A_{0})
		= \sigma_{0}(I),
	\end{equation}
	which means $A_{0} - I \in \Hpsido^{-1}(M)$.
\end{proof}

The proof of the following proposition is inspired by
that of~\cite{Beals-Greiner1988}*{Proposition 25.4}.

\begin{proposition}
	There exist $\Pi_{\infty} \in \Hpsido^{0}(M)$ and $G_{\infty} \in \Hpsido^{- 4}(M)$ such that
	\begin{gather}
		G_{\infty} \GJMS + \Pi_{\infty}
		\sim P G_{\infty} + \Pi_{\infty}
		= I, \\
		G_{\infty}^{\ast}
		\sim G_{\infty},
		\qquad
		\Pi_{\infty}^{\ast}
		\sim \Pi_{\infty}^{2}
		\sim \Pi_{\infty},
		\qquad
		\Pi_{\infty} - \Pi_{0}
		\in \Hpsido^{- 1}(M), \\
		\Pi_{\infty} \GJMS
		\sim \GJMS \Pi_{\infty}
		\sim 0,
		\qquad
		\Pi_{\infty} G_{\infty}
		\sim G_{\infty} \Pi_{\infty}
		\sim 0.
	\end{gather}
\end{proposition}

\begin{proof}
	Let $A_{0} \in \Hpsido^{0}(M)$ be a parametrix of $I + R_{0}$,
	and set
	\begin{equation}
		G_{\infty}
		\coloneqq (I - \Pi_{0} A_{0}) G_{0} A_{0}
		\in \Hpsido^{- 4}(M),
		\qquad
		\Pi_{\infty}
		\coloneqq I - \GJMS G_{\infty}
		\in \Hpsido^{0}(M).
	\end{equation}
	First note that
	\begin{equation}
		\Pi_{\infty}
		= I - (\GJMS - \GJMS \Pi_{0} A_{0}) G_{0} A_{0}
		\sim (I + R_{0} - \GJMS G_{0}) A_{0}
		\sim \Pi_{0} A_{0}.
	\end{equation}
	In particular,
	$\GJMS \Pi_{\infty} \sim \GJMS \Pi_{0} A_{0} \sim 0$
	and $\Pi_{\infty} - \Pi_{0} \in \Hpsido^{- 1}(M)$.
	These yield
	\begin{align}
		\Pi_{\infty}^{\ast}
		= \Pi_{\infty}^{\ast} (\GJMS G_{\infty} + \Pi_{\infty})
		\sim \Pi_{\infty}^{\ast} \Pi_{\infty}
		\sim (G_{\infty}^{\ast} \GJMS + \Pi_{\infty}^{\ast}) \Pi_{\infty}
		= \Pi_{\infty}.
	\end{align}
	We also have
	\begin{equation}
		\Pi_{\infty} G_{\infty}
		\sim \Pi_{\infty} (I - \Pi_{\infty}) G_{0} A_{0}
		= (\Pi_{\infty} - \Pi_{\infty}^{2}) G_{0} A_{0}
		\sim 0
	\end{equation}
	and
	\begin{align}
		G_{\infty} \Pi_{\infty}
		= (G_{\infty}^{\ast} \GJMS + \Pi_{\infty}^{\ast}) G_{\infty} \Pi_{\infty}
		&\sim (G_{\infty}^{\ast} \GJMS G_{\infty} + \Pi_{\infty} G_{\infty}) \Pi_{\infty} \\
		&\sim G_{\infty}^{\ast} (I - \Pi_{\infty}) \Pi_{\infty} \\
		&= G_{\infty}^{\ast} (\Pi_{\infty} - \Pi_{\infty}^{2}) \\
		&\sim 0.
	\end{align}
	Therefore
	\begin{equation}
		G_{\infty}^{\ast}
		\sim G_{\infty}^{\ast} (I - \Pi_{\infty})
		= G_{\infty}^{\ast} \GJMS G_{\infty}
		\sim (G_{\infty}^{\ast} \GJMS + \Pi_{\infty}) G_{\infty}
		= G_{\infty},
	\end{equation}
	which completes the proof.
\end{proof}

Consider $\GJMS$ as an unbounded closed operator acting on $L^{2}(M)$
by the maximal closed extension.
The domain $\Dom \GJMS$ contains $\HSobolev^{4}(M)$ by \cref{prop:mapping-properties-of-Hpsido}.
Conversely,
any $u \in \Dom \GJMS$
is an element of $\HSobolev^{4}(M)$ modulo $\Ran \Pi_{\infty}$
by the lemma below.

\begin{lemma}
\label{lem:domain-of-critical-GJMS-operator}
	For $u \in \Dom \GJMS$,
	one has $u - \Pi_{\infty} u \in \HSobolev^{4}(M)$.
	In particular,
	$\Dom \GJMS = \Ran \Pi_{\infty} + \HSobolev^{4}(M)$.
\end{lemma}

\begin{proof}
	Set
	\begin{equation} \label{eq:left-parametrix}
		R_{\infty}
		\coloneqq G_{\infty} \GJMS + \Pi_{\infty} - I
		\in \Hpsido^{- \infty}(M).
	\end{equation}
	If $v = \GJMS u \in L^{2}(M)$,
	then
	\begin{equation}
		u - \Pi_{\infty} u
		= G_{\infty} v - R_{\infty} u
		\in \HSobolev^{4}(M).
	\end{equation}
	In particular,
	$u \in \Ran \Pi_{\infty} + \HSobolev^{4}(M)$.
	Moreover,
	$\GJMS \Pi_{\infty} \sim 0$ implies $\Ran \Pi_{\infty} \subset \Dom \GJMS$.
	Therefore we have $\Dom \GJMS = \Ran \Pi_{\infty} + \HSobolev^{4}(M)$.
\end{proof}

\begin{proof}[Proof of \cref{thm:ess-self-adj-of-CR-Paneitz}]
	It suffices to show that $\GJMS$ is symmetric.
	Let $u, v \in \Dom \GJMS$.
	It follows from \cref{lem:domain-of-critical-GJMS-operator} that
	$v^{\prime} \coloneqq v - \Pi_{\infty} v$ is in $\HSobolev^{4}(M)$.
	Since $P \Pi_{\infty} \sim 0$,
	\begin{equation}
		\iproduct{\GJMS u}{\Pi_{\infty} v}_{0}
		= \iproduct{\Pi_{\infty}^{\ast} \GJMS u}{v}_{0}
		= \iproduct{(\GJMS \Pi_{\infty})^{\ast} u}{v}_{0}
		= \iproduct{u}{\GJMS \Pi_{\infty} v}_{0}.
	\end{equation}
	Take a sequence $(v_{j})$ in $C^{\infty}(M)$
	such that $v_{j}$ converges to $v^{\prime}$ in $\HSobolev^{4}(M)$ as $j \to + \infty$.
	Then $\GJMS v_{j}$ converges to $\GJMS v^{\prime}$ in $L^{2}(M)$ as $j \to + \infty$
	by the continuity of $\GJMS \colon \HSobolev^{4}(M) \to L^{2}(M)$.
	Thus we have
	\begin{equation}
		\iproduct{\GJMS u}{v^{\prime}}_{0}
		= \lim_{j \to \infty} \iproduct{\GJMS u}{v_{j}}_{0}
		= \lim_{j \to \infty} \iproduct{u}{\GJMS v_{j}}_{0}
		= \iproduct{u}{\GJMS v^{\prime}}_{0}.
	\end{equation}
	Therefore $\iproduct{\GJMS u}{v}_{0} = \iproduct{u}{\GJMS v}_{0}$,
	which means that $\GJMS$ is symmetric.
\end{proof}

Let $E$ be the resolution of the identity for $\GJMS$ and fix $\lambda \geq 0$.
Set
\begin{equation}
	\pi_{\lambda} \coloneqq E(\clcl{- \lambda}{\lambda})
	\colon L^{2}(M) \to \Dom \GJMS.
\end{equation}
This is an orthogonal projection of $L^{2}(M)$ and satisfies
$\pi_{\lambda} \GJMS = \GJMS \pi_{\lambda}$ on $\Dom \GJMS$.
If $\lambda > 0$,
\begin{equation}
	N_{\lambda} \coloneqq \int_{\bbR} t^{- 1} \chi_{\clcl{- \lambda}{\lambda}^{c}}(t) \, d E(t)
	\colon L^{2}(M) \to \Dom \GJMS
\end{equation}
is a continuous self-adjoint operator and satisfies
\begin{gather}
	P N_{\lambda} + \pi_{\lambda} = I \text{\ on \ } L^{2}(M), \\
	N_{\lambda} P + \pi_{\lambda} = I \text{\ on \ } \Dom \GJMS.
\end{gather}

\begin{theorem}
\label{thm:smoothing-property-of-projection}
	If $\lambda > 0$,
	then $\GJMS \pi_{\lambda} \sim 0$.
\end{theorem}

\begin{proof}
	We first show that $\GJMS \pi_{\lambda}$ defines a continuous map
	from $L^{2}(M)$ to $\HSobolev^{4 k}(M)$ for any $k \in \bbZ_{\geq 0}$.
	It follows from $\GJMS G_{\infty} + \Pi_{\infty} = I$ that
	\begin{equation}
	\label{eq:111}
		\GJMS \pi_{\lambda}
		= (G_{\infty}^{\ast} \GJMS + \Pi_{\infty}^{\ast}) \GJMS \pi_{\lambda}
		= G_{\infty}^{\ast} (\GJMS \pi_{\lambda})^{2} + \Pi_{\infty}^{\ast} P \pi_{\lambda}.
	\end{equation}
	Since $\Pi_{\infty}^{\ast} P \sim \Pi_{\infty} P \sim 0$,
	the second term maps $L^{2}(M)$ to $\HSobolev^{4 k}(M)$ continuously for any $k \in \bbZ_{\geq 0}$.
	If $\GJMS \pi_{\lambda}$ maps $L^{2}(M)$ to $\HSobolev^{4 k}(M)$ continuously,
	then $G_{\infty}^{\ast} (\GJMS \pi_{\lambda})^{2}$ maps $L^{2}(M)$ to $\HSobolev^{4 k + 4}(M)$ continuously.
	Hence $\GJMS \pi_{\lambda} \colon L^{2}(M) \to \HSobolev^{4 k + 4}(M)$ is continuous.
	We obtain by induction that $\GJMS \pi_{\lambda}$ is a continuous operator
	from $L^{2}(M)$ to $\HSobolev^{4 k}(M)$ for any $k \in \bbZ_{\geq 0}$.
	Taking the adjoint yields that $\GJMS \pi_{\lambda}$ extends to a continuous operator
	from $\HSobolev^{- 4 l}(M)$ to $L^{2}(M)$ for any $l \in \bbZ_{\geq 0}$.

	We next show that $\pi_{\lambda}$ extends to a continuous operator acting on $\HSobolev^{- 4 l}(M)$
	for any $l \in \bbZ_{\geq 0}$.
	It follows from $\GJMS N_{\lambda} + \pi_{\lambda} = I$ that
	\begin{equation}
		\Pi_{\infty}^{\ast}
		= \Pi_{\infty}^{\ast} \GJMS N_{\lambda} + \Pi_{\infty}^{\ast} \pi_{\lambda}.
	\end{equation}
	Similarly,
	$\GJMS G_{\infty} + \Pi_{\infty} = I$ yields that
	\begin{equation}
		\pi_{\lambda}
		= G_{\infty}^{\ast} \GJMS \pi_{\lambda} + \Pi_{\infty}^{\ast} \pi_{\lambda}.
	\end{equation}
	Thus we have
	\begin{equation}
		\pi_{\lambda} - \Pi_{\infty}^{\ast}
		= G_{\infty}^{\ast} \GJMS \pi_{\lambda} - \Pi_{\infty}^{\ast} \GJMS N_{\lambda},
	\end{equation}
	which maps $L^{2}(M)$ to $\HSobolev^{4 k}(M)$ continuously for any $k \in \bbZ_{\geq 0}$.
	Taking the adjoint gives that
	$\pi_{\lambda} - \Pi_{\infty}$ extends to a continuous operator
	from $\HSobolev^{- 4 l}(M)$ to $L^{2}(M)$ for any $l \in \bbZ_{\geq 0}$.
	Since $\Pi_{\infty}$ defines a continuous operator acting on $\HSobolev^{- 4 l}(M)$,
	so does $\pi_{\lambda}$.
	This and \cref{eq:111} imply that $\GJMS \pi_{\lambda}$ extends to a continuous operator
	from $\HSobolev^{- 4 l}(M)$ to $\HSobolev^{4 k}(M)$ for any $k, l \in \bbZ_{\geq 0}$,
	which means that it is a smoothing operator.
\end{proof}

\begin{proof}[Proof of \cref{thm:spectrum-of-CR-Paneitz}]
	Fix $\lambda > 0$.
	Then $\Spec \GJMS \cap \clcl{- \lambda}{\lambda} = \Spec \GJMS \pi_{\lambda}$ and
	\begin{equation}
		\Ker (\GJMS - \mu I)
		= \Ker (\GJMS \pi_{\lambda} - \mu I)
	\end{equation}
	for any $0 < \abs{\mu} \leq \lambda$.
	On the other hand,
	it follows from \cref{thm:smoothing-property-of-projection} that
	$\GJMS \pi_{\lambda}$ is a compact self-adjoint operator acting on $L^{2}(M)$,
	and so $\Spec \GJMS \pi_{\lambda}$ is discrete except $0$.
	Moreover,
	$\Ker (\GJMS \pi_{\lambda} - \mu I)$ is a finite dimensional subspace of $C^{\infty}(M)$
	for any $\mu \neq 0$.
	These completes the proof.
\end{proof}

\begin{theorem}
\label{thm:spectral-projection-is-Hpsido}
	If $\lambda > 0$,
	then $\pi_{\lambda}$ and $N_{\lambda}$
	are Heisenberg pseudodifferential operators of order $0$ and $- 4$ respectively.
	Moreover,
	$\pi_{\lambda}$ and $N_{\lambda}$ coincide with $\Pi_{\infty}$ and $G_{\infty}$ respectively
	modulo $\Hpsido^{- \infty}(M)$.
\end{theorem}

\begin{proof}
	It follows from $N_{\lambda} \GJMS + \pi_{\lambda} = I$ that
	\begin{equation}
		N_{\lambda} \GJMS \Pi_{\infty} + \pi_{\lambda} \Pi_{\infty}
		= \Pi_{\infty}.
	\end{equation}
	Taking the adjoint yields that
	\begin{equation}
		\Pi_{\infty}^{\ast} \GJMS N_{\lambda} + \Pi_{\infty}^{\ast} \pi_{\lambda}
		= \Pi_{\infty}^{\ast}.
	\end{equation}
	Thus we have
	\begin{equation}
		(\Pi_{\infty}^{\ast} - \Pi_{\infty}^{\ast} \pi_{\lambda})
		(\Pi_{\infty} - \pi_{\lambda} \Pi_{\infty})
		= \Pi_{\infty}^{\ast} \GJMS N_{\lambda}^{2} \GJMS \Pi_{\infty},
	\end{equation}
	which is a smoothing operator.
	On the other hand,
	\begin{align}
		&(\Pi_{\infty}^{\ast} - \Pi_{\infty}^{\ast} \pi_{\lambda})
			(\Pi_{\infty} - \pi_{\lambda} \Pi_{\infty}) \\
		&= (\Pi_{\infty}^{\ast} + G_{\infty}^{\ast} \GJMS \pi_{\lambda} - \pi_{\lambda})
			(\Pi_{\infty} + P \pi_{\lambda} G_{\infty} - \pi_{\lambda}) \\
		&= \Pi_{\infty}^{\ast} \Pi_{\infty} + \Pi_{\infty}^{\ast} \GJMS \pi_{\lambda} G_{\infty}
			- \Pi_{\infty}^{\ast} \pi_{\lambda} + G_{\infty}^{\ast} \GJMS \pi_{\lambda} \Pi_{\infty} \\
		&\quad + G_{\infty}^{\ast} (\GJMS \pi_{\lambda})^{2} G_{\infty} - G_{\infty}^{\ast} \GJMS \pi_{\lambda}^{2}
			- \pi_{\lambda} \Pi_{\infty} + \GJMS \pi_{\lambda} G_{\infty} + \pi_{\lambda}^{2} \\
		&\sim \Pi_{\infty} - \Pi_{\infty}^{\ast} \pi_{\lambda} - \pi_{\lambda} \Pi_{\infty} + \pi_{\lambda} \\
		&= \Pi_{\infty} + G_{\infty}^{\ast} \GJMS \pi_{\lambda} - \pi_{\lambda}
			+ \GJMS \pi_{\lambda} G_{\infty} - \pi_{\lambda} + \pi_{\lambda} \\
		&\sim \Pi_{\infty} - \pi_{\lambda}.
	\end{align}
	In particular,
	$\pi_{\lambda}$ is a Heisenberg pseudodifferential operator of order $0$.

	Next consider $N_{\lambda}$.
	It follows from $\GJMS G_{\infty} + \Pi_{\infty} = I$ that
	\begin{equation}
		(I - \pi_{\lambda}) G_{\infty} + N_{\lambda} \Pi_{\infty} = N_{\lambda},
		\qquad
		G_{\infty}^{\ast} (I - \pi_{\lambda}) + \Pi_{\infty}^{\ast} N_{\lambda} = N_{\lambda}.
	\end{equation}
	Hence
	\begin{align}
		N_{\lambda} - G_{\infty}
		&= N_{\lambda} (\Pi_{\infty} - \pi_{\lambda}) - \pi_{\lambda} G_{\infty} \\
		&= (\Pi_{\infty}^{\ast} - \pi_{\lambda}) N_{\lambda} (\Pi_{\infty} - \pi_{\lambda})
			+ G_{\infty}^{\ast} (I - \pi_{\lambda}) (\Pi_{\infty} - \pi_{\lambda}) \\
		&\quad + (\Pi_{\infty} - \pi_{\lambda}) G_{\infty} - \Pi_{\infty} G_{\infty},
	\end{align}
	which is a smoothing operator.
	Therefore $N_{\lambda}$ is a Heisenberg pseudodifferential operator of order $- 4$.
\end{proof}

\section{CR Paneitz operator on the Rossi sphere}
\label{section:CR-Paneitz-operator-on-Rossi-spheres}

In this section,
we prove the existence of infinitely many negative eigenvalues
of the CR Paneitz operator on the Rossi sphere.
Our proof is inspired by that of \cite{Abbas-Brown-Ramasami-Zeytuncu2019}*{Theorem 5.7}.

\subsection{Definition of the Rossi sphere}
\label{subsection:definition-of-Rossi-spheres}

The unit sphere
\begin{equation}
	S^{3}
	\coloneqq \Set{(z, w) \in \bbC^{2} | \abs{z}^{2} + \abs{w}^{2} = 1}
\end{equation}
has the canonical CR structure $T^{1, 0} S^{3}$.
This CR structure is spanned by
\begin{equation}
	Z_{1}
	\coloneqq \ovw \frac{\del}{\del z} - \ovz \frac{\del}{\del w}.
\end{equation}
A canonical contact form $\theta$ on $S^{3}$ is given by
\begin{equation}
	\theta
	\coloneqq \frac{\sqrt{- 1}}{2} (z d \ovz + w d \ovw
		- \ovz d z - \ovw d w)|_{S^{3}}.
\end{equation}
The Reeb vector field $T$ with respect to $\theta$ is written as
\begin{equation}
	T
	= \sqrt{- 1} \rbra*{z \frac{\del}{\del z} + w \frac{\del}{\del w}
		- \ovz \frac{\del}{\del \ovz} - \ovw \frac{\del}{\del \ovw} }.
\end{equation}
The admissible coframe corresponding to $(T, Z_{1}, Z_{\ovone})$ is given by
\begin{equation}
	(\theta, \theta^{1} \coloneqq (w d z - z d w)|_{S^{3}},
	\theta^{\ovone} \coloneqq \overline{\theta^{1}}).
\end{equation}
For this coframe,
we have
\begin{equation}
	d \theta
	= \sqrt{-1} \theta^{1} \wedge \theta^{\ovone},
	\qquad
	d \theta^{1}
	= 2 \sqrt{- 1} \theta \wedge \theta^{1}.
\end{equation}
These imply
\begin{equation}
	l_{1 \ovone}
	= 1,
	\qquad
	A_{1 1}
	= 0,
	\qquad
	{\omega_{1}}^{1}
	= - 2 \sqrt{-1} \theta,
	\qquad
	\Scal
	= 2.
\end{equation}
In particular,
\begin{equation}
\label{eq:Kohn-Laplacian-on-standard-sphere}
	\Box_{b}
	= - l^{1 \ovone} (Z_{1} Z_{\ovone} - {\omega_{\ovone}}^{\ovone}(Z_{1}) Z_{\ovone})
	= - Z_{1} Z_{\ovone}.
\end{equation}

For a real number $0 < \abs{t} < 1$,
the \emph{Rossi sphere} $(S^{3}_{t}, T^{1, 0} S^{3}_{t})$ is defined by
\begin{equation}
	(S^{3}_{t}, T^{1, 0} S^{3}_{t})
	\coloneqq (S^{3}, \bbC (Z_{1} + t Z_{\ovone})).
\end{equation}
The complex vector field
\begin{equation}
	Z_{1}(t)
	\coloneqq Z_{1} + t Z_{\ovone}
\end{equation}
gives a global frame of $T^{1, 0} S^{3}_{t}$.
The admissible coframe with respect to
$(T, Z_{1}(t), Z_{\ovone}(t) \coloneqq \overline{Z_{1}(t)})$
is of the form
\begin{equation}
	(\theta,
	\theta^{1}(t)
	\coloneqq (1 - t^{2})^{- 1}(\theta^{1} - t \theta^{\ovone}),
	\theta^{\ovone}(t)
	\coloneqq \overline{\theta^{1}(t)}).
\end{equation}
For this coframe,
we have
\begin{gather}
	d \theta
	= \sqrt{- 1} (1 - t^{2}) \theta^{1}(t) \wedge \theta^{\ovone}(t), \\
	d \theta^{1}(t)
	= \sqrt{- 1} \frac{2(1 + t^{2})}{1 - t^{2}} \theta \wedge \theta^{1}(t)
		+ \sqrt{- 1} \frac{4 t}{1 - t^{2}} \theta \wedge \theta^{\ovone}(t).
\end{gather}
This implies
\begin{equation}
	l_{1 \ovone}(t)
	= 1 - t^{2},
	\qquad
	A^{1 1}(t)
	= \sqrt{- 1} \frac{4 t}{(1 - t^{2})^{2}},
	\qquad
	{\omega_{1}}^{1}(t)
	= - \sqrt{- 1} \frac{2 (1 + t^{2})}{1 - t^{2}} \theta;
\end{equation}
see~\cite{Chanillo-Chiu-Yang2012}*{Proposition 2.5} for example.

\begin{lemma}
\label{lem:Kohn-Laplacian-on-Rossi-sphere}
	The Kohn Laplacian $\Box_{b}(t)$ and $\calQ(t)$
	with respect to $(S^{3}_{t}, T^{1, 0} S^{3}_{t}, \theta)$
	satisfy
	\begin{gather}
		(1 - t^{2}) \Box_{b}(t)
		= \Box_{b} - t (Z_{1})^{2} - t (Z_{\ovone})^{2} + t^{2} \overline{\Box}_{b}, \\
		(1 - t^{2})^{2} \calQ(t)
		= 4 t (Z_{\ovone})^{2}  - 4 t^{2} (\Box_{b} + \overline{\Box}_{b}) + 4 t^{3} (Z_{1})^{2}.
	\end{gather}
\end{lemma}

\begin{proof}
	It follows from the definition of the Kohn Laplacian and \cref{eq:Kohn-Laplacian-on-standard-sphere} that
	\begin{align}
		(1 - t^{2}) \Box_{b}(t)
		&= - (1 - t^{2}) l^{1 \ovone}(t)
				\rbra*{Z_{1}(t) Z_{\ovone}(t)
				- {\omega_{\ovone}}^{\ovone}(t) (Z_{1}(t)) Z_{\ovone}(t)} \\
		&= - Z_{1}(t) Z_{\ovone}(t) \\
		&= \Box_{b} - t (Z_{1})^{2} - t (Z_{\ovone})^{2} + t^{2} \overline{\Box}_{b}.
	\end{align}
	Similarly,
	$(1 - t^{2})^{2} \calQ(t)$ is given by
	\begin{align}
		(1 - t^{2})^{2} \calQ(t)
		&= \sqrt{- 1} (1 - t^{2})^{2} \rbra*{Z_{\ovone}(t)
			+ {\omega_{\ovone}}^{\ovone}(t)(Z_{\ovone}(t))}
				\rbra*{A^{\ovone \ovone}(t) Z_{\ovone}(t)} \\
		&= 4 t (Z_{\ovone}(t))^{2} \\
		&= 4 t (Z_{\ovone})^{2}  - 4 t^{2} (\Box_{b} + \overline{\Box}_{b}) + 4t ^{3} (Z_{1})^{2},
	\end{align}
	which completes the proof.
\end{proof}

\subsection{Spherical harmonics}
\label{subsection:spherical-harmonics}

In this subsection,
we recall some facts on spherical harmonics,
which give a good orthogonal decomposition of $L^{2}(S^{3})$.
We denote by $\scrP_{p, q}(\bbC^{2})$
the space of complex homogeneous polynomials of bidegree $(p, q)$
and by $\scrH_{p, q}(\bbC^{2})$ the space of harmonic $f \in \scrP_{p, q}(\bbC^{2})$.
The restriction map
\begin{equation}
	|_{S^{3}} \colon \scrH_{p, q}(\bbC^{2}) \to C^{\infty}(S^{3})
\end{equation}
is injective since any $f \in \scrH_{p,q}$ is harmonic.
The image of $\scrH_{p, q}(\bbC^{2})$ under this map
is written as $\scrH_{p, q}(S^{3})$.

\begin{lemma}[c.f.\ \cite{Abbas-Brown-Ramasami-Zeytuncu2019}*{Propositions 2.1 and 2.2}]
\label{lem:orthogonal-decomposition-of-L^2}
	The Hilbert space $L^{2}(S^{3})$ has the following orthogonal decomposition:
	\begin{equation}
		L^{2}(S^{3})
		= \bigoplus_{p, q} \scrH_{p, q}(S^{3}).
	\end{equation}
	Moreover,
	$\dim \scrH_{p, q}(S^{3}) = p + q + 1$;
	in particular,
	$\scrH_{p, q}(S^{3})$ contains a non-zero element.
\end{lemma}

Moreover,
any $f \in \scrH_{p, q}(S^{3})$ is a simultaneous eigenfunction of $\Box_{b}$ and $\overline{\Box}_{b}$.

\begin{lemma}[c.f.\ \cite{Abbas-Brown-Ramasami-Zeytuncu2019}*{Theorem 2.6}]
\label{lem:eigenvalue-of-Kohn-Laplacian}
	For any $f \in \scrH_{p, q}(S^{3})$,
	one has
	\begin{equation}
		\Box_{b} f
		= (p + 1) q f,
		\qquad
		\overline{\Box}_{b} f
		= p (q + 1) f.
	\end{equation}
\end{lemma}

Furthermore,
$Z_{1}$ and $Z_{\ovone}$ change the bidegree.

\begin{lemma}
	The vector fields $Z_{1}$ and $Z_{\ovone}$ map
	$\scrH_{p, q}(S^{3})$ to $\scrH_{p - 1, q + 1}(S^{3})$ and $\scrH_{p + 1, q - 1}(S^{3})$ respectively.
\end{lemma}

\begin{proof}
	It follows from the definition of $Z_{1}$ and $Z_{\ovone}$ that
	\begin{equation}
		Z_{1} \colon \scrP_{p, q}(\bbC^{2}) \to \scrP_{p - 1, q + 1}(\bbC^{2}),
		\qquad
		Z_{\ovone} \colon \scrP_{p, q}(\bbC^{2}) \to \scrP_{p + 1, q - 1}(\bbC^{2}).
	\end{equation}
	Moreover,
	one can check that $\comm{\Delta}{Z_{1}} = \comm{\Delta}{Z_{\ovone}} = 0$,
	where $\Delta$ is the Euclidean Laplacian on $\bbC^{2}$.
	Hence $Z_{1}$ and $Z_{\ovone}$ map harmonic functions to those.
\end{proof}

\subsection{Negative eigenvalues of the CR Paneitz operator}

Set $\calP(t) \coloneqq (1 - t^{2})^{2} P(t)$,
where $P(t)$ is the CR Paneitz operator with respect to $(S^{3}_{t}, T^{1,0}S^{3}_{t}, \theta)$.
It suffices to show that $\calP(t)$ has infinitely many negative eigenvalues.

Fix a positive integer $k$
and take $v_{1} \in \scrH_{2 k - 1, 0}(S^{3})$ with $\norm{v_{1}}_{L^{2}} = 1$;
the existence of such $v_{1}$ follows from \cref{lem:orthogonal-decomposition-of-L^2}.
We set
\begin{equation}
	c_{k}(l)
	\coloneqq (l - 2) (2 k - l + 2)
\end{equation}
and
\begin{equation}
	v_{i}
	\coloneqq \rbra*{\prod_{l = 1}^{i - 1} \frac{1}{\sqrt{c_{k}(2 l + 1) c_{k}(2 l + 2)}}}
		(Z_{1})^{2 i - 2} v_{1}
		\in \scrH_{2 k - 2 i + 1, 2 i - 2}
\end{equation}
for each integer $2 \leq i \leq k$.
Note that
\begin{equation}
\label{eq:relation-of-coefficient}
	c_{k}(l) + c_{k}(l + 3)
	= c_{k}(l + 1) + c_{k}(l + 2) - 4.
\end{equation}
In what follows,
we use the convention $v_{i} = 0$ for $i \leq 0$ or $i \geq k + 1$.

\begin{lemma}
	The functions $(v_{i})$ satisfy
	\begin{gather}
		\Box_{b} v_{i}
		= c_{k}(2 i) v_{i},
		\qquad
		\overline{\Box}_{b} v_{i}
		= c_{k}(2 i + 1) v_{i}, \\
		(Z_{1})^{2} v_{i}
		= \sqrt{c_{k}(2 i + 1) c_{k}(2 i + 2)} v_{i + 1},
		\qquad
		(Z_{\ovone})^{2} v_{i}
		= \sqrt{c_{k}(2 i - 1) c_{k}(2 i)} v_{i - 1}.
	\end{gather}
\end{lemma}

\begin{proof}
	Since $v_{i} \in \scrH_{2 k - 2 i + 1, 2 i - 2}$,
	\cref{lem:eigenvalue-of-Kohn-Laplacian} implies the first and second equalities.
	The third one easily follows from the definition of $v_{i}$.
	Moreover,
	if $i \geq 2$,
	then
	\begin{align}
		(Z_{\ovone})^{2} v_{i}
		&= \frac{1}{\sqrt{c_{k}(2 i - 1) c_{k}(2 i)}} (Z_{\ovone})^{2} (Z_{1})^{2} v_{i - 1} \\
		&= - \frac{1}{\sqrt{c_{k}(2 i - 1) c_{k}(2 i)}} Z_{\ovone} \overline{\Box}_{b} (Z_{1} v_{i - 1}) \\
		&= \frac{\sqrt{c_{k}(2 i)}}{\sqrt{c_{k}(2 i - 1)}} \overline{\Box}_{b} v_{i - 1} 
			\qquad \text{($\because$ $Z_{1} v_{i - 1} \in \scrH_{2 k - 2 i + 2, 2 i - 3}$)} \\
		&= \sqrt{c_{k}(2 i - 1) c_{k}(2 i)} v_{i - 1},
	\end{align}
	which completes the proof.
\end{proof}

Let $V_{k}$ be the subspace of $L^{2}(S^{3})$
spanned by $(v_{i})_{i = 1}^{k}$.
This space is closed under $\Box_{b}$, $\overline{\Box}_{b}$, $(Z_{1})^{2}$, and $(Z_{\ovone})^{2}$.

\begin{lemma}
	The family $(v_{i})_{i = 1}^{k}$ is an orthonormal basis of $V_{k}$.
\end{lemma}

\begin{proof}
	Since $v_{i} \in \scrH_{2 k - 2 i + 1, 2 i - 2}(S^{3})$,
	it is sufficient to show $\norm{v_{i}}_{L^{2}} = 1$ for $1 \leq i \leq k$.
	We prove this by induction.
	It follows from the choice of $v_{1}$ that $\norm{v_{1}}_{L^{2}} = 1$.
	Assume that $\norm{v_{i}}_{L^{2}} = 1$ for some $1 \leq i \leq k - 1$.
	Then the integration by parts gives
	\begin{align}
		\norm{v_{i + 1}}_{L^{2}}^{2}
		&= \frac{1}{\sqrt{c_{k}(2 i + 1) c_{k}(2 i + 2)}} \int_{S^{3}} v_{i + 1}
				\rbra*{\overline{(Z_{1})^{2} v_{i}}} \, \theta \wedge d \theta \\
		&= \frac{1}{\sqrt{c_{k}(2 i + 1) c_{k}(2 i + 2)}} \int_{S^{3}} \rbra*{(Z_{\ovone})^{2} v_{i + 1}}
				\overline{v_{i}} \, \theta \wedge d \theta \\
		&= \int_{S^{3}} \abs{v_{i}}^{2} \, \theta \wedge d \theta \\
		&= 1,
	\end{align}
	which completes the proof.
\end{proof}

It follows from the definition of the CR Paneitz operator
and \cref{lem:Kohn-Laplacian-on-Rossi-sphere} that
\begin{align}
	&\calP(t) v_{i} \\
	&= t^{2} \sqrt{c_{k}(2 i - 3) c_{k}(2 i - 2) c_{k}(2 i - 1) c_{k}(2 i)} v_{i - 2} \\
	&\quad - t (1 + t^{2}) (c_{k}(2 i - 2) + c_{k}(2 i + 1)) \sqrt{c_{k}(2 i - 1) c_{k}(2 i)} v_{i - 1} \\
	&\quad + \sbra*{(1 + t^{2})^{2} c_{k}(2 i) c_{k}(2 i + 1)
		+ t^{2} (c_{k}(2 i - 2) c_{k}(2 i) + c_{k}(2 i + 1) c_{k}(2 i + 3))} v_{i} \\
	&\quad - t (1 + t^{2}) (c_{k}(2 i) + c_{k}(2 i + 3)) \sqrt{c_{k}(2 i + 1) c_{k}(2 i + 2)} v_{i + 1} \\
	&\quad + t^{2} \sqrt{c_{k}(2 i + 1) c_{k}(2 i + 2) c_{k}(2 i + 3) c_{k}(2 i + 4)}v_{i + 2};
\end{align}
here we use the equality \cref{eq:relation-of-coefficient}.
In particular,
this yields that
$\calP(t)$ maps $V_{k}$ to itself.
We would like to show $\calP(t)|_{V_{k}}$ has exactly one negative eigenvalue.
We denote by $\calP_{k}(t)$
the matrix representation of $\calP(t)|_{V_{k}}$ with respect to $(v_{i})_{i = 1}^{k}$.
Note that $\calP_{k}(t)$ is a $k \times k$ Hermitian matrix
since $\calP(t)$ is self-adjoint and $(v_{i})_{i = 1}^{k}$ is an orthonormal basis of $V_{k}$.

\begin{proposition}
\label{prop:exactly-one-negative-eigenvalue}
	The Hermitian matrix $\calP_{k}(t)$ has exactly one negative eigenvalue.
\end{proposition}

\begin{proof}
	We define an $l \times l$ submatrix $\calP_{k, l}(t)$ of $\calP_{k}(t)$ for each $1 \leq l \leq k$ by
	\begin{equation}
		(\calP_{k, l}(t))_{i, j}
		\coloneqq (\calP_{k}(t))_{i, j},
		\qquad (i, j = 1, \dots, l),
	\end{equation}
	and set $\eta_{k, l}(t) \coloneqq \det \calP_{k, l}(t)$.
	In order to compute $\eta_{k, l}(t)$,
	we apply some elementary row operations.
	Add inductively the $i$-th row multiplied by $(1 + t^{2})\sqrt{c_{k}(2 i + 2) / c_{k}(2 i + 1)} / t$
	(resp.\ $- \sqrt{(c_{k}(2 i + 2) c_{k}(2 i + 4)) / (c_{k}(2 i + 1) c_{k}(2 i + 3))}$)
	to the $(i + 1)$-st row (resp.\ $(i + 2)$-nd row),
	which does not change the determinant.
	The resulting matrix is the upper triangular matrix given by
	\begin{equation}
		\begin{pmatrix}
			t^{2} c_{k}(3) c_{k}(5) & & & \text{\Huge{*}} \\
			& t^{2} c_{k}(5) c_{k}(7) & & \\
			& & \ddots & & \\
			\text{\Huge{0}} & & & t^{2} c_{k}(2 l + 1) c_{k}(2 l + 3)
		\end{pmatrix}
	\end{equation}
	Hence
	\begin{equation}
		\eta_{k, l}(t)
		= t^{2 l} c_{k}(3) c_{k}(2 l + 3) \prod_{i = 1}^{l - 1} c_{k}(2 i + 3)^{2}.
	\end{equation}
	In particular,
	$\eta_{k, l}(t)$ is positive (resp.\ negative)
	if $1 \leq l \leq k - 1$ (resp.\ $l = k$).
	Applying Cauchy's interlace theorem inductively yields
	that $\calP_{k, l}(t)$ has only positive eigenvalues for $1 \leq l \leq k - 1$
	and $\calP_{k, k}(t) = \calP_{k}(t)$ has exactly one negative eigenvalue.
\end{proof}

\begin{comment}
\begin{proof}
	We prove \cref{eq:minor-of-whcalP} by induction.
	If $l = 1$,
	then
	\begin{equation}
		\eta_{k, 1}(t) = (\widehat{\calP}_{k, 1}(t))_{1, 1} = t^{2} c_{k}(3) c_{k}(5).
	\end{equation}
	Suppose that \cref{eq:minor-of-whcalP} holds for integers smaller than or equal to $l$.
	It is known that
	$\eta_{l}(t)$ satisfies the following recurrence relation:
	\begin{align}
		\eta_{k, l + 1}(t)
		&= t^{2} [c_{k}(2 l + 2) + c_{k}(2 l + 5)] c_{k}(2 l + 3) \eta_{k, l}(t) \\
		&\quad - t^{4} c_{k}(2 l + 1) c_{k}(2 l + 2) c_{k}(2 l + 3)^{2} \eta_{k, l - 1}(t)
	\end{align}
	with initial values $\eta_{k, 0}(t) = 1$ and $\eta_{k, -1}(t) = 0$;
	this is because $\widehat{\calP}_{k, l}(t)$ is a tridiagonal matrix.
	It follows from the assumption that
	\begin{align}
		\eta_{k, l + 1}(t)
		&= t^{2 l + 2} [c_{k}(2 l + 2) + c_{k}(2 l + 5)] c_{k}(2 l + 3) c_{k}(3) c_{k}(2 l + 3)
			\prod_{i = 1}^{l - 1} c_{k}(2 i + 3)^{2} \\
		&\quad - t^{2 l + 2} c_{k}(2 l + 1) c_{k}(2 l + 2) c_{k}(2 l + 3)^{2} c_{k}(3) c_{k}(2 l + 1)
			\prod_{i = 1}^{l - 2} c_{k}(2 i + 3)^{2} \\
		&= t^{2 l + 2} c_{k}(3) c_{k}(2 l + 5) \prod_{i = 1}^{l} c_{k}(2 i + 3)^{2},
	\end{align}
	which implies that \cref{eq:minor-of-whcalP} is true for $l + 1$,
	which completes the proof of \cref{eq:minor-of-whcalP}.
	Since $c_{k}(2 i + 3)$ is positive (resp.\ negative) for $0 \leq i \leq k - 1$ (resp.\ $i = k$),
	we have the latter statement.
\end{proof}
\end{comment}

\begin{proof}[Proof of \cref{thm:infinitely-many-negative-eigenvalue}]
	As we proved above,
	there exists an eigenfunction $0 \neq f_{k} \in V_{k}$ of $P(t)$ with negative eigenvalue
	for each positive integer $k$.
	Since
	\begin{equation}
		V_{k}
		\subset \bigoplus_{p + q = 2 k - 1} \scrH_{p, q}(S^{3}),
		\qquad
		L^{2}(S^{3}) = \bigoplus_{p, q} \scrH_{p, q}(S^{3}),
	\end{equation}
	the family $(f_{k})_{k = 1}^{\infty}$ is linearly independent.
	This yields that $P(t)$ has infinitely many negative eigenvalues \emph{with} multiplicity.
	On the other hand,
	it follows from \cref{thm:spectrum-of-CR-Paneitz} that
	$\Spec P(t) \setminus \{0\}$ consists only of eigenvalues of finite multiplicity.
	Therefore $P(t)$ has infinitely many negative eigenvalues \emph{without} multiplicity.
\end{proof}

\begin{remark}
\label{rem:ortho-proj-to-Ker-is-not-Hpsido}
	The proof of \cref{prop:exactly-one-negative-eigenvalue} implies that
	the kernel of the operator
	\begin{equation}
		P(t) \colon \bigoplus_{p + q = 2 k - 1} \scrH_{p, q}
			\to \bigoplus_{p + q = 2 k - 1} \scrH_{p, q}
	\end{equation}
	is equal to zero for any positive integer $k$.
	In particular,
	any function annihilated by $P(t)$ must be even.
	Hence the Schwartz kernel of the orthogonal projection $\pi_{0}(t)$ to $\Ker P(t)$ has the singularity on
	\begin{equation}
		\Set{((z, w), \pm (z, w)) \in S^{3} \times S^{3}},
	\end{equation}
	and so $\pi_{0}(t)$ can not be a Heisenberg pseudodifferential operator.
\end{remark}

\section{Concluding remarks}
\label{section:concluding-remarks}

The author~\cite{Takeuchi2020-Paneitz}*{Theorem 1.1} has proved that
the CR Paneitz operator on any \emph{embeddable} CR manifold is non-negative.
On the other hand,
we found that the CR Paneitz operator on the Rossi sphere has infinitely many negative eigenvalues
in the previous section.
It is natural to ask whether this phenomenon occurs
in general \emph{non-embeddable} CR manifolds.

\begin{problem}
	Does the CR Paneitz operator on any \emph{non-embeddable} CR manifold
	have necessarily infinitely many negative eigenvalues?
\end{problem}

Moreover,
Hsiao~\cite{Hsiao2015}*{Theorem 4.7} has shown that
the CR Paneitz operator on any \emph{embeddable} CR manifold has closed range;
see also~\cite{Takeuchi2023-GJMS} for another proof via Heisenberg calculus.
On the other hand,
\cref{thm:infinitely-many-negative-eigenvalue} only asserts that
there are infinitely many negative eigenvalues of $P(t)$,
and it does not make any claims about the distribution of those eigenvalues.

\begin{problem}
	Does the CR Paneitz operator on the Rossi sphere have closed range?
\end{problem}

Using Mathematica for calculations,
we find that the following holds for small values of $k$:
\begin{align}
	\det(\calP_{1}(t) + 3 t^{2} I)
	&= 0, \\
	\det(\calP_{2}(t) + 3 t^{2} I)
	&= 36 t^{2} (1 - t^{2})^{2}, \\
	\det(\calP_{3}(t) + 3 t^{2} I)
	&= 576 t^{2} (1 - t^{2})^{2} (15 + 58 t^{2} + 15 t^{4}), \\
	\det(\calP_{4}(t) + 3 t^{2} I)
	&= 6480 t^{2} (1 - t^{2})^{2} (1680 + 6549 t^{2} + 15926 t^{4} + 6549 t^{6} + 1680 t^{8}), \\
	\det(\calP_{5}(t) + 3 t^{2} I)
	&= 995328 t^{2} (1 - t^{2})^{2} (44100 + 172683 t^{2} + 422712 t^{4} + 825970 t^{6} \\
	&\quad + 422712 t^{8} + 172683 t^{10} + 44100 t^{12}), \\
	\det(\calP_{6}(t) + 3 t^{2} I)
	&= 4536000 t^{2} (1 - t^{2})^{2} (95800320 + 376277184 t^{2} + 924268539 t^{4} \\
	&\quad + 1815582548 t^{6} + 3114137570 t^{8} + 1815582548 t^{10} \\
	&\quad + 924268539 t^{12} + 376277184 t^{14} + 95800320 t^{16}).
\end{align}
These calculations suggest that the determinant of $\calP_{k}(t) + 3 t^{2} I$
is non-negative for any $- 1 < t < 1$.
If this is true,
the unique negative eigenvalue of $\calP_{k}(t)$ (\cref{prop:exactly-one-negative-eigenvalue})
is bigger than or equal to $- 3 t^{2}$,
and so the spectrum of $P(t)$ has $0$ as an accumulation point;
in particular,
the CR Paneitz operator on the Rossi sphere does not have closed range.

\bibliography{my-reference,my-reference-preprint}

\begin{bibdiv}
\begin{biblist}

\bib{Abbas-Brown-Ramasami-Zeytuncu2019}{article}{
      author={Abbas, Tawfik},
      author={Brown, Madelyne~M.},
      author={Ramasami, Allison},
      author={Zeytuncu, Yunus~E.},
       title={Spectrum of the {K}ohn {L}aplacian on the {R}ossi sphere},
        date={2019},
        ISSN={1944-4176},
     journal={Involve},
      volume={12},
      number={1},
       pages={125\ndash 140},
         url={https://doi.org/10.2140/involve.2019.12.125},
      review={\MR{3810483}},
}

\bib{Boutet_de_Monvel1975}{incollection}{
      author={Boutet~de Monvel, L.},
       title={Int\'egration des \'equations de {C}auchy-{R}iemann induites
  formelles},
        date={1975},
   booktitle={S\'eminaire {G}oulaouic-{L}ions-{S}chwartz 1974--1975;
  \'equations aux deriv\'ees partielles lin\'eaires et non lin\'eaires},
   publisher={Centre Math., \'Ecole Polytech., Paris},
       pages={Exp. No. 9, 14},
      review={\MR{0409893}},
}

\bib{Beals-Greiner1988}{book}{
      author={Beals, Richard},
      author={Greiner, Peter},
       title={Calculus on {H}eisenberg manifolds},
      series={Annals of Mathematics Studies},
   publisher={Princeton University Press, Princeton, NJ},
        date={1988},
      volume={119},
        ISBN={0-691-08500-5; 0-691-08501-3},
         url={http://dx.doi.org/10.1515/9781400882397},
      review={\MR{953082}},
}

\bib{Chanillo-Chiu-Yang2012}{article}{
      author={Chanillo, Sagun},
      author={Chiu, Hung-Lin},
      author={Yang, Paul},
       title={Embeddability for 3-dimensional {C}auchy-{R}iemann manifolds and
  {CR} {Y}amabe invariants},
        date={2012},
        ISSN={0012-7094},
     journal={Duke Math. J.},
      volume={161},
      number={15},
       pages={2909\ndash 2921},
         url={https://doi.org/10.1215/00127094-1902154},
      review={\MR{2999315}},
}

\bib{Cheng-Malchiodi-Yang2017}{article}{
      author={Cheng, Jih-Hsin},
      author={Malchiodi, Andrea},
      author={Yang, Paul},
       title={A positive mass theorem in three dimensional {C}auchy-{R}iemann
  geometry},
        date={2017},
        ISSN={0001-8708},
     journal={Adv. Math.},
      volume={308},
       pages={276\ndash 347},
         url={https://doi.org/10.1016/j.aim.2016.12.012},
      review={\MR{3600060}},
}

\bib{Chen-Shaw2001}{book}{
      author={Chen, So-Chin},
      author={Shaw, Mei-Chi},
       title={Partial differential equations in several complex variables},
      series={AMS/IP Studies in Advanced Mathematics},
   publisher={American Mathematical Society, Providence, RI; International
  Press, Boston, MA},
        date={2001},
      volume={19},
        ISBN={0-8218-1062-6},
         url={https://doi.org/10.1090/amsip/019},
      review={\MR{1800297}},
}

\bib{Fefferman-Hirachi2003}{article}{
      author={Fefferman, Charles},
      author={Hirachi, Kengo},
       title={Ambient metric construction of {$Q$}-curvature in conformal and
  {CR} geometries},
        date={2003},
        ISSN={1073-2780},
     journal={Math. Res. Lett.},
      volume={10},
      number={5-6},
       pages={819\ndash 831},
         url={http://dx.doi.org/10.4310/MRL.2003.v10.n6.a9},
      review={\MR{2025058}},
}

\bib{Gover-Graham2005}{article}{
      author={Gover, A.~Rod},
      author={Graham, C.~Robin},
       title={C{R} invariant powers of the sub-{L}aplacian},
        date={2005},
        ISSN={0075-4102},
     journal={J. Reine Angew. Math.},
      volume={583},
       pages={1\ndash 27},
         url={http://dx.doi.org/10.1515/crll.2005.2005.583.1},
      review={\MR{2146851}},
}

\bib{Hirachi1993}{incollection}{
      author={Hirachi, Kengo},
       title={Scalar pseudo-{H}ermitian invariants and the {S}zeg{\H{o}} kernel
  on three-dimensional {CR} manifolds},
        date={1993},
   booktitle={Complex geometry ({O}saka, 1990)},
      series={Lecture Notes in Pure and Appl. Math.},
      volume={143},
   publisher={Dekker, New York},
       pages={67\ndash 76},
      review={\MR{1201602}},
}

\bib{Hsiao-Marinescu2017}{article}{
      author={Hsiao, Chin-Yu},
      author={Marinescu, George},
       title={On the singularities of the {S}zeg{\H{o}} projections on lower
  energy forms},
        date={2017},
        ISSN={0022-040X,1945-743X},
     journal={J. Differential Geom.},
      volume={107},
      number={1},
       pages={83\ndash 155},
         url={https://doi.org/10.4310/jdg/1505268030},
      review={\MR{3698235}},
}

\bib{Hsiao2015}{article}{
      author={Hsiao, Chin-Yu},
       title={On {CR} {P}aneitz operators and {CR} pluriharmonic functions},
        date={2015},
        ISSN={0025-5831},
     journal={Math. Ann.},
      volume={362},
      number={3-4},
       pages={903\ndash 929},
         url={http://dx.doi.org/10.1007/s00208-014-1151-2},
      review={\MR{3368087}},
}

\bib{Kohn1986}{article}{
      author={Kohn, J.~J.},
       title={The range of the tangential {C}auchy-{R}iemann operator},
        date={1986},
        ISSN={0012-7094},
     journal={Duke Math. J.},
      volume={53},
      number={2},
       pages={525\ndash 545},
         url={http://dx.doi.org/10.1215/S0012-7094-86-05330-5},
      review={\MR{850548}},
}

\bib{Lee1988}{article}{
      author={Lee, John~M.},
       title={Pseudo-{E}instein structures on {CR} manifolds},
        date={1988},
        ISSN={0002-9327},
     journal={Amer. J. Math.},
      volume={110},
      number={1},
       pages={157\ndash 178},
         url={http://dx.doi.org/10.2307/2374543},
      review={\MR{926742}},
}

\bib{Ponge2008-Book}{article}{
      author={Ponge, Rapha\"el~S.},
       title={Heisenberg calculus and spectral theory of hypoelliptic operators
  on {H}eisenberg manifolds},
        date={2008},
        ISSN={0065-9266},
     journal={Mem. Amer. Math. Soc.},
      volume={194},
      number={906},
       pages={viii+ 134},
         url={http://dx.doi.org/10.1090/memo/0906},
      review={\MR{2417549}},
}

\bib{Takeuchi2020-Paneitz}{article}{
      author={Takeuchi, Yuya},
       title={Nonnegativity of the {CR} {P}aneitz operator for embeddable {CR}
  manifolds},
        date={2020},
        ISSN={0012-7094},
     journal={Duke Math. J.},
      volume={169},
      number={18},
       pages={3417\ndash 3438},
         url={https://doi.org/10.1215/00127094-2020-0051},
      review={\MR{4181029}},
}

\bib{Takeuchi2023-GJMS}{article}{
      author={Takeuchi, Yuya},
       title={Analysis of the critical {CR} {GJMS} operator},
        date={2023},
        ISSN={0002-9327,1080-6377},
     journal={Amer. J. Math.},
      volume={145},
      number={6},
       pages={1923\ndash 1940},
      review={\MR{4684293}},
}

\end{biblist}
\end{bibdiv}

\end{document}